\documentclass[11pt]{article}

\usepackage[margin=1.1in]{geometry}
\usepackage{amsmath,amssymb,amsthm,mathtools}
\usepackage{enumitem}
\usepackage{xcolor}
\usepackage[colorlinks=true,linkcolor=blue!50!black,citecolor=blue!50!black,urlcolor=blue!50!black]{hyperref}
\usepackage[noabbrev]{cleveref}

\setlist{itemsep=0.2em,topsep=0.35em}

\theoremstyle{plain}
\newtheorem{theorem}{Theorem}[section]
\newtheorem{proposition}{Proposition}[section]
\newtheorem{lemma}{Lemma}[section]

\theoremstyle{definition}
\newtheorem{definition}{Definition}[section]
\theoremstyle{remark}

\crefname{theorem}{Theorem}{Theorems}
\crefname{proposition}{Proposition}{Propositions}
\crefname{lemma}{Lemma}{Lemmas}
\crefname{corollary}{Corollary}{Corollaries}
\crefname{definition}{Definition}{Definitions}
\crefname{remark}{Remark}{Remarks}
\crefname{section}{Section}{Sections}

\newcommand{\N}{\mathbb N}
\newcommand{\Z}{\mathbb Z}
\newcommand{\E}{\mathbb E}
\newcommand{\Prob}{\mathbb P}
\newcommand{\floor}[1]{\left\lfloor #1\right\rfloor}
\newcommand{\avg}{\operatorname{av}}

\title{Long Lattice Paths with No Three Collinear Vertices}
\author{Samuel Korsky}
\date{\today}

\begin{document}
\maketitle

\begin{abstract}
\noindent
For \(d\ge 1\), let \(L(d)\in\N\cup\{\infty\}\) be the supremum of the lengths of paths in \(\Z^d\) whose steps are standard basis vectors and whose vertex sets contain no collinear triple. We prove that
\[
        \log_2\log_2 L(d)\ge \frac{2}{5}d-O(1)
\]
for all sufficiently large \(d\).
\end{abstract}

\section{Introduction}

A \emph{standard-basis path} in \(\Z^d\) is a sequence
\[
        P_0,P_1,\ldots,P_n
\]
such that \(P_t-P_{t-1}\in\{e_1,\ldots,e_d\}\) for \(1\le t\le n\). Let \(L(d)\in\N\cup\{\infty\}\) be the supremum of the possible lengths of such paths whose vertex sets contain no collinear triple; logarithms of \(+\infty\) are interpreted as \(+\infty\).

The condition that no three visited vertices are collinear is a path-constrained version of the classical no-three-in-line problem, in which one selects a large subset of an integer grid with no three selected points collinear \cite{GuyKelly1968,BrassMoserPach2005}. The path constraint is substantial: the vertices must occur in an order for which every increment is one of the prescribed positive coordinate directions.

The planar north--east case has a separate history. Brown asked whether, for every fixed \(k\), each sufficiently long path with steps \((1,0)\) and \((0,1)\) must contain \(k\) collinear vertices \cite{Brown1971}; Montgomery proved that it must \cite{Montgomery1972}. Gerver and Ramsey gave an explicit upper bound on the required length, while Gerver constructed long paths avoiding \(k\) collinear vertices \cite{Gerver1979,GerverRamsey1979}. Recent work has determined the small values of \(k\) computationally and improved the asymptotic bounds \cite{BarnoffBright2026,Korsky2026}. These planar results keep the dimension equal to two and let the forbidden number of collinear vertices grow.

The finite-step walk problem also has a higher-dimensional side. Gerver and Ramsey constructed an infinite walk in \(\Z^3\), using finitely many allowed steps, whose intersection with every line is uniformly bounded; Lidbetter later improved the numerical bound for that construction \cite{GerverRamsey1979,Lidbetter2024}. Here we study a complementary regime: the forbidden configuration is always a collinear triple, the allowed steps are the standard basis vectors, and the dimension grows.

There is an exact formulation in terms of words. Encode the step \(e_i\) by the letter \(i\), so a path is represented by a word \(x_1\cdots x_n\) over \([d]=\{1,\ldots,d\}\). For a word \(W\), let \(c(W)\) be its letter-count vector, also called its Parikh vector. If \(i<j<k\), then \(P_i,P_j,P_k\) are collinear exactly when
\[
 \frac{c(x_{i+1}\cdots x_j)}{j-i}
 =
 \frac{c(x_{j+1}\cdots x_k)}{k-j}.
\]
Thus we seek long words with no two adjacent nonempty factors having the same normalized count vector. Classical abelian squares require adjacent factors of equal length with the same letter-count vector; their avoidance has been extensively studied \cite{Dekking1979,Keranen1992,FiciPuzynina2023}. Our obstruction permits unequal lengths and compares normalized counts, so it is closer to weak abelian periodicity, where successive blocks have the same letter-frequency vector \cite{AvgustinovichPuzynina2016}. Below we call such an adjacent pair a \emph{bad pair} and call a word with no bad pair \emph{3-free}.

Our main result is the following.

\begin{theorem}\label{thm:main}
There is an absolute constant \(C_0>0\) such that, for all sufficiently large \(d\),
\begin{equation}\label{eq:main-loglog}
        \log_2\log_2 L(d)\ge \frac{2}{5}d-C_0.
\end{equation}
\end{theorem}

Equivalently, \(L(d)\) is at least doubly exponential in \(d\), with inner exponent \(\frac25d-O(1)\). A projection argument also gives a finite planar consequence: an \(n\)-step walk with no collinear triple can be generated by only \(\frac52\log_2\log_2 n+O(1)\) allowed steps. The allowed step set may depend on \(n\), so this does not resolve the infinite fixed-step-set problem \cite{ErdosProblems193}.

The proof repeatedly turns one 3-free word into a much longer one. Start with a 3-free word \(X\) containing a letter \(\sigma\) with positive density. Append a new separator letter and repeat the resulting word periodically. The separator gives an exact reduction: a bad pair in the periodic word can occur only when both factor lengths are whole multiples of the period.

Next, group the occurrences of \(\sigma\) into blocks and replace \(\sigma\) by \(p+1\) new letters. In each block, the first \(p\) replacement counts are independently perturbed while their total remains fixed. For a possible bad pair with lengths \(a\) and \(b\) periods, remaining badness imposes \(p\) independent weighted-sum equations. The point-probability bounds in \cref{lem:local-gaussian,lem:weighted-uniform-sums} show that, for \(p\ge3\), the probabilities are summable over all \((a,b)\). The resulting estimate makes the expected number of bad pairs less than one, so some choice of the random replacements creates none.

Quantitatively, a word of length \(n\) with a marked letter of density at least \(\rho\) yields one of length at least \(c_p\rho^pn^p\). The alphabet grows by at most \(p+1\), and a marked letter of density at least \(c_p\rho\) remains. Iteration gives the rate
\[
        \frac{\log_2p}{p+1}.
\]
Among integers \(p\ge3\), this rate is uniquely maximized at \(p=4\), yielding the coefficient \(2/5\).

\Cref{sec:words} develops the word formulation and the periodic reduction. \Cref{sec:probability} proves the point-probability estimates, \cref{sec:amplification} gives the amplification step, and \cref{sec:iteration} carries out the iteration. \Cref{sec:planar} proves the finite planar consequence.

\section{Words and a Periodic Reduction}\label{sec:words}

Let \(\Sigma\) be a finite alphabet. A word \(X=x_1\cdots x_n\) over \(\Sigma\) determines a standard-basis path by
\begin{equation}\label{eq:path-from-steps}
        v_0(X)=0,
        \qquad
        v_t(X)=\sum_{m=1}^t e_{x_m}
        \quad (1\le t\le n).
\end{equation}
For a nonempty interval \(I=[a+1,b]\) of positions, define
\begin{equation}\label{eq:average-increment}
        |I|=b-a,
        \qquad
        \avg_X(I)=\frac{v_b(X)-v_a(X)}{b-a}.
\end{equation}
The coordinates of \(\avg_X(I)\) are the relative letter counts in the factor indexed by \(I\). Two adjacent nonempty intervals \(I=[i+1,j]\) and \(J=[j+1,k]\) form a \emph{bad pair} if
\begin{equation}\label{eq:bad-pair}
        \avg_X(I)=\avg_X(J).
\end{equation}
We call \(X\) \emph{3-free} if its path contains no three collinear vertices.

For \(i<j<k\), the two successive displacement vectors are
\[
        (j-i)\avg_X([i+1,j])
        \quad\text{and}\quad
        (k-j)\avg_X([j+1,k]).
\]
They have nonnegative coordinates and positive coordinate sums, so the three vertices are collinear exactly when the displacements are positive scalar multiples. Dividing by their coordinate sums shows that this is equivalent to equality of the two averages. Thus a word is 3-free if and only if it has no bad pair.

Write \(|W|_\sigma\) for the number of occurrences of a letter \(\sigma\) in a word \(W\).

\begin{definition}\label{def:marked-word}
Let $0<\rho\le1$. A triple $(X,\sigma,\rho)$ is a \emph{marked 3-free word} if $X$ is 3-free and
\begin{equation}\label{eq:marked-density}
        |X|_\sigma\ge \rho|X|.
\end{equation}
\end{definition}

Only this global density is needed by the construction. This allows us to construct a \emph{initial word} of arbitrarily large fixed length:

\begin{proposition}[Seed]\label{prop:seed}
For every prescribed \(Q\ge1\), there is a marked 3-free word \((X,\sigma,1/2)\) with \(|X|\ge Q\).
\end{proposition}

\begin{proof}
Choose \(m\) with \(2m\ge Q\), take pairwise distinct letters \(a_1,\ldots,a_m\) different from \(\sigma\), and set
\begin{equation}\label{eq:simple-seed}
        X=\sigma a_1\sigma a_2\cdots\sigma a_m.
\end{equation}
Then \(|X|_\sigma=m=|X|/2\). If a bad pair contained some \(a_i\), equality of normalized counts would force \(a_i\) to occur in both adjacent intervals, which is impossible because \(a_i\) occurs only once. Hence both intervals would have to contain only \(\sigma\), also impossible because no two occurrences of \(\sigma\) are adjacent. Thus \(X\) is 3-free.
\end{proof}

The separator gives an exact periodic reduction.

\begin{lemma}[Periodic reduction]\label{lem:periodic-reduction}
Let \(X\) be a 3-free word of length \(n\), let \(\#\) be a new letter, and let
\[
        W=(X\#)^\infty
\]
have period \(M=n+1\). Two adjacent nonempty intervals of \(W\) form a bad pair if and only if both lengths are multiples of \(M\).
\end{lemma}

\begin{proof}
Put \(F_r=v_r(X)\) for \(0\le r\le n\), viewed in the enlarged coordinate space with \(\#\)-coordinate \(0\), and set
\[
        V=F_n+e_\#.
\]
After \(qM+r\) steps, where \(q\ge0\) and \(0\le r<M\), the path generated by \(W\) is at \(qV+F_r\).

Write the three endpoints of two adjacent intervals as
\[
        u=q_0M+r,\qquad
        v=q_1M+s,\qquad
        w=q_2M+t,
\]
where \(0\le r,s,t<M\), and put
\[
        \alpha=q_1-q_0,\qquad \beta=q_2-q_1.
\]
Since \(u<v<w\), we have \(\alpha,\beta\ge0\). The two displacement vectors are
\[
        \alpha V+F_s-F_r,
        \qquad
        \beta V+F_t-F_s.
\]
If their averages are equal, comparison of the \(\#\)-coordinate gives
\begin{equation}\label{eq:separator-ratio}
        \frac{\alpha}{v-u}=\frac{\beta}{w-v}.
\end{equation}
If \(\alpha=0\), then \(\beta=0\). In that case \(r<s<t\), so the two intervals form a bad pair inside \(X\), a contradiction. Thus \(\alpha,\beta>0\).

Using \eqref{eq:separator-ratio} in the remaining coordinates cancels the full-period contribution and gives
\begin{equation}\label{eq:prefix-collinearity}
        \beta(F_s-F_r)=\alpha(F_t-F_s),
\end{equation}
or equivalently
\[
        (\alpha+\beta)F_s=\beta F_r+\alpha F_t.
\]
The coordinate sum of \(F_j\) is \(j\), so \(F_0,\ldots,F_n\) are distinct. If \(r,s,t\) are pairwise distinct, the last display makes three distinct vertices of \(X\) collinear. If exactly two residues are equal, \eqref{eq:prefix-collinearity} forces the third to be equal as well. Hence \(r=s=t\), and the two interval lengths are \(\alpha M\) and \(\beta M\).

Conversely, every interval of length \(mM\), with \(m\ge1\), has displacement \(mV\), independently of its starting position, and therefore has average \(V/M\).
\end{proof}

\section{Point-Probability Bounds for Random Sums}\label{sec:probability}

We first bound the largest point probability of a sum of discrete uniform variables.

\begin{lemma}[Local Gaussian bound]\label{lem:local-gaussian}
Let $h,q\ge1$, let $U_1,\ldots,U_q$ be independent and uniform on $\{-h,\ldots,h\}$, and put $S_q=\sum_{j=1}^qU_j$. There are absolute constants $c,C>0$ such that, for every $x\in\Z$,
\begin{equation}\label{eq:local-gaussian}
        \Prob(S_q=x)
        \le \frac{C}{h\sqrt q}\cdot
        \exp\left(-c\cdot\frac{x^2}{h^2q}\right).
\end{equation}
\end{lemma}

\begin{proof}
Write
\[
        p_q(x)=\Prob(S_q=x).
\]
The distribution of each $U_j$ is symmetric and log-concave on $\Z$. Since
convolution preserves both properties, the distribution of $S_q$ is also
symmetric and log-concave, and hence unimodal. In particular,
$p_q(x)$ is nonincreasing for $x\ge0$.

We first prove the uniform point-mass estimate
\begin{equation}\label{eq:max-uniform-sum}
        \max_{x\in\Z}p_q(x)\le \frac{C}{h\sqrt q}.
\end{equation}
For $q=1$, this follows immediately from
$p_1(x)\le(2h+1)^{-1}$. Suppose that $q\ge2$. The characteristic function
of one summand is
\[
        \phi_h(\theta)
        =\E \left[e^{i\theta U_1}\right]
        =\frac{1}{2h+1}\sum_{u=-h}^h e^{iu\theta}
        =\frac{\sin((2h+1)\theta/2)}
               {(2h+1)\sin(\theta/2)}.
\]
By independence, the characteristic function of $S_q$ is
$\phi_h(\theta)^q$. Fourier inversion therefore gives
\[
        p_q(x)
        =\frac{1}{2\pi}
          \int_{-\pi}^{\pi}
          e^{-ix\theta}\phi_h(\theta)^q\,d\theta,
\]
and hence
\[
        p_q(x)
        \le \frac{1}{2\pi}
          \int_{-\pi}^{\pi}|\phi_h(\theta)|^q\,d\theta.
\]

We estimate this integral separately near and away from the origin. For
$|\theta|\le c_0/h$, the usual Taylor estimate for the sine quotient gives
\[
        |\phi_h(\theta)|
        \le \exp(-c h^2\theta^2).
\]
Consequently,
\[
        \int_{|\theta|\le c_0/h}
        |\phi_h(\theta)|^q\,d\theta
        \le
        \int_{\mathbb R}
        \exp(-cqh^2\theta^2)\,d\theta
        \le \frac{C}{h\sqrt q}.
\]

On the complementary range \(c_0/h\le|\theta|\le\pi\), the explicit formula
for \(\phi_h\) gives
\[
        |\phi_h(\theta)|\le \eta<1
        \qquad\text{and}\qquad
        |\phi_h(\theta)|\le \frac{C}{h|\theta|}
\]
for absolute constants \(c_0>0\) and \(\eta<1\). The uniform gap from \(1\) can be seen by splitting at \(A/h\): after the rescaling \(y=h\theta\), the range \(c_0/h\le|\theta|\le A/h\) is compact and excludes \(0\), while on \(|\theta|\ge A/h\) the second bound is \(<1\) when \(A\) is large. Since \(q\ge2\),
\[
\begin{aligned}
        \int_{c_0/h\le|\theta|\le\pi}
        |\phi_h(\theta)|^q\,d\theta
        &\le
        C\eta^{q-2}
        \int_{c_0/h}^{\pi}
        \frac{d\theta}{h^2\theta^2}  \\
        &\le \frac{C\eta^{q-2}}{h}
        \le \frac{C}{h\sqrt q}.
\end{aligned}
\]
where the last inequality uses that exponential decay in \(q\) dominates \(q^{-1/2}\). Together these estimates prove \eqref{eq:max-uniform-sum}.

We now add the Gaussian decay in $x$. Since the variables are independent,
have mean zero, and satisfy $|U_j|\le h$, Hoeffding's inequality \cite{Hoeffding1963} gives
\begin{equation}\label{eq:uniform-hoeffding}
        \Prob(|S_q|\ge y)
        \le
        2\exp\left(-c\cdot\frac{y^2}{h^2q}\right)
\end{equation}
for every $y\ge0$.

First suppose that $|x|\le2h\sqrt q$. In this range,
\[
        \exp\left(-c\cdot\frac{x^2}{h^2q}\right)
        \ge e^{-4c},
\]
so the desired estimate follows directly from
\eqref{eq:max-uniform-sum}, after increasing the constant $C$.

It remains to consider the tail. By symmetry, it is enough to assume that
$x>2h\sqrt q$. Since $p_q$ is nonincreasing on the nonnegative integers,
\[
        p_q(y)\ge p_q(x)
\]
for every integer $y$ with $\lceil x/2\rceil\le y\le x$. Summing over these
values gives
\[
        \floor{x/2}\,p_q(x)
        \le \Prob(S_q\ge x/2).
\]
Applying \eqref{eq:uniform-hoeffding}, and absorbing the factor arising
from $(x/2)^2$ into the constant $c$, yields
\[
        p_q(x)
        \le
        \frac{C}{x}\cdot
        \exp\left(-c\cdot\frac{x^2}{h^2q}\right).
\]
Because $x>2h\sqrt q$, we have $x^{-1}\le C/(h\sqrt q)$, and therefore
\[
        p_q(x)
        \le
        \frac{C}{h\sqrt q}\cdot
        \exp\left(-c\cdot\frac{x^2}{h^2q}\right).
\]
The case $x<0$ follows by symmetry, completing the proof.
\end{proof}

We also need a point-probability bound for a weighted difference of two such sums.

\begin{lemma}[Weighted point-mass bound]\label{lem:weighted-uniform-sums}
Let $h,q,r,a,b\ge1$. Let $U_1,\ldots,U_q,V_1,\ldots,V_r$ be independent and uniform on $\{-h,\ldots,h\}$, and put
\[
        S_q=\sum_{j=1}^qU_j,
        \qquad
        T_r=\sum_{j=1}^rV_j.
\]
If $D=\gcd(a,b)$, then
\begin{equation}\label{eq:weighted-uniform-sums}
 \sup_{z\in\Z}\Prob(bS_q-aT_r=z)
 \le C\left(
        \frac{1}{h^2\sqrt{qr}}
        +\frac{D}{h\sqrt{a^2r+b^2q}}
 \right)
\end{equation}
for an absolute constant $C$.
\end{lemma}

\begin{proof}
Write $a=Da_0$ and $b=Db_0$, where $\gcd(a_0,b_0)=1$. If $D\nmid z$, the probability is zero. Otherwise, the integer solutions of
\[
        b_0x-a_0y=z/D
\]
are
\[
        x=x_0+a_0k,
        \qquad
        y=y_0+b_0k,
        \qquad k\in\Z,
\]
for one fixed solution $(x_0,y_0)$. By \cref{lem:local-gaussian},
\begin{align*}
 \Prob(bS_q-aT_r=z)
 &\le \frac{C}{h^2\sqrt{qr}}
 \sum_{k\in\Z}
 \exp\left[-c\left(
   \frac{(x_0+a_0k)^2}{h^2q}
   +\frac{(y_0+b_0k)^2}{h^2r}
 \right)\right].
\end{align*}
After completing the square, the quadratic expression in $k$ has leading coefficient
\[
        A=\frac{a_0^2r+b_0^2q}{h^2qr}.
\]
The Gaussian sum bound
\[
        \sup_{\theta\in\mathbb R}
        \sum_{k\in\Z}e^{-cA(k-\theta)^2}
        \le C(1+A^{-1/2}),
\]
which follows by comparing the sum with \(1+\int_{\mathbb R}e^{-cAx^2}\,dx\), therefore gives
\[
 \Prob(bS_q-aT_r=z)
 \le C\left(
        \frac{1}{h^2\sqrt{qr}}
        +\frac{1}{h\sqrt{a_0^2r+b_0^2q}}
 \right),
\]
which is \eqref{eq:weighted-uniform-sums}.
\end{proof}

\section{Amplification by Randomly Splitting One Letter}\label{sec:amplification}

We state the construction for a general number \(p\) of independent perturbations.

\begin{proposition}[Amplification]\label{prop:amplification}
Fix an integer $p\ge3$. There are constants $c_p,C_p>0$ with the following property. Let $(X,\sigma,\rho)$ be a marked 3-free word of length $n$ using \(r\) letters. If
\begin{equation}\label{eq:amplifier-size-hypothesis}
        n\ge C_p\rho^{-2},
\end{equation}
then there is a marked 3-free word \((Y,\tau,c_p\rho)\) using at most \(r+p+1\) letters such that
\begin{equation}\label{eq:amplifier-conclusion}
        c_p\rho^p n^p\le |Y|\le n^p.
\end{equation}
\end{proposition}

\begin{proof}
During the proof, \(c_p\) may be decreased and \(C_p\) increased; only finitely many such changes are made, so their final values are fixed. We may also assume \(0<c_p\le1\). Put
\[
        s=|X|_\sigma\ge\rho n,
        \qquad
        M=n+1,
\]
and form \(W=(X\#)^\infty\), where \(\#\) is a new letter. The size hypothesis implies
\[
        s\ge C_p\rho^{-1}
        \quad\text{and}\quad
        \frac{s^p}{M}
        \ge \frac{\rho^pn^{p-1}}{2}
        \ge \frac{C_p^{p-1}}{2}\cdot\rho^{2-p}
        \ge \frac{C_p^{p-1}}{2}.
\]
Thus, by increasing \(C_p\), we may assume that every later lower bound depending only on \(p\) is satisfied.

\emph{Random split.}
Choose a fixed integer \(K=K(p)\ge4\). Number the occurrences of \(\sigma\) in \(W\) as \(1,2,\ldots\), and let each \(\sigma\)-block consist of
\[
        g=\floor{s/K}
\]
consecutive occurrences in this ordering. Put
\[
        u=\floor{g/(p+1)},
        \qquad
        h=\floor{u/(10p)}.
\]
For each \(\sigma\)-block \(B\), independently choose
\[
        U_{B,1},\ldots,U_{B,p}
\]
uniformly from \(\{-h,\ldots,h\}\). Replace the \(g\) occurrences of \(\sigma\) in \(B\) by new letters \(\sigma_1,\ldots,\sigma_{p+1}\) with respective counts
\begin{equation}\label{eq:block-composition}
        u+U_{B,1},\ldots,u+U_{B,p},
        \qquad
        g-pu-\sum_{i=1}^pU_{B,i}.
\end{equation}
Assign the first prescribed number of \(\sigma\)-positions in \(B\) to \(\sigma_1\), the next prescribed number to \(\sigma_2\), and so on.

Write \(g=(p+1)u+r_0\), where \(0\le r_0\le p\). The first \(p\) counts in \eqref{eq:block-composition} are at least \(u-h\), and the last is at least \(u-ph\). For sufficiently large \(u\), all are therefore at least \(\delta_pg\) for some \(\delta_p>0\). Moreover,
\[
        g\ge \frac{s}{2K},\qquad
        u\ge \frac{g}{2(p+1)},\qquad
        h\ge \frac{u}{20p},
\]
so
\begin{equation}\label{eq:h-comparable-s}
        h\ge c_ps.
\end{equation}

Choose \(0<\varepsilon_p\le1\) sufficiently small and set
\begin{equation}\label{eq:amplifier-length}
        N=M\floor{\frac{\varepsilon_ps^p}{M}}.
\end{equation}
By the preceding lower bound on \(s^p/M\), we may assume that the quantity inside the floor is at least \(2\). Since \(\lfloor x\rfloor\ge x/2\) for \(x\ge2\),
\begin{equation}\label{eq:N-bounds}
        \frac{\varepsilon_p}{2}\cdot s^p\le N\le \varepsilon_ps^p\le n^p.
\end{equation}
Let \(Y\) be the first \(N\) letters of the split word.

\emph{Possible bad pairs.}
Any bad pair in \(Y\) projects, after merging \(\sigma_1,\ldots,\sigma_{p+1}\) back to \(\sigma\), to a bad pair in \(W\). By \cref{lem:periodic-reduction}, its two lengths are \(aM\) and \(bM\) for positive integers \(a,b\). For each fixed pair \((a,b)\), there are at most \(N\) possible starting positions.

Fix one possible bad pair, with adjacent intervals \(I\) and \(J\) of lengths \(aM\) and \(bM\). Each interval of length \(aM\) in \(W\) contains exactly \(as\) occurrences of \(\sigma\), regardless of its starting position. Let \(q_I\) and \(q_J\) be the numbers of complete \(\sigma\)-blocks contained in \(I\) and \(J\), respectively. At most two \(\sigma\)-blocks meet an interval without being contained in it. Since \(g\le s/K\),
\begin{equation}\label{eq:complete-block-counts}
        q_I\ge \frac{as}{g}-2\ge (K-2)a,
        \qquad
        q_J\ge \frac{bs}{g}-2\ge (K-2)b.
\end{equation}

Condition on every random variable except those attached to complete \(\sigma\)-blocks contained in \(I\) or \(J\). If the pair remains bad after the split, equality of the normalized counts of \(\sigma_i\), for each \(1\le i\le p\), imposes an equation
\begin{equation}\label{eq:coordinate-collision}
        b\sum_{B\subset I}U_{B,i}
        -a\sum_{B\subset J}U_{B,i}
        =z_i,
\end{equation}
where \(z_i\) is determined by the conditioned variables, the baseline counts, and the partially intersected blocks. The \(p\) equations use independent families of random variables. The last replacement letter gives no further condition, because its count is determined by the first \(p\) counts and the total number of replaced occurrences.

Let \(D=\gcd(a,b)\). From \eqref{eq:h-comparable-s} and \eqref{eq:complete-block-counts},
\[
 h^2\sqrt{q_Iq_J}\ge c_ps^2\sqrt{ab},
 \qquad
 a^2q_J+b^2q_I\ge c_pab(a+b).
\]
Applying \cref{lem:weighted-uniform-sums} therefore shows that the conditional probability of one equation in \eqref{eq:coordinate-collision} is at most
\begin{equation}\label{eq:one-coordinate-survival}
        C_p\left(
        \frac{1}{s^2\sqrt{ab}}
        +\frac{D}{s\sqrt{ab(a+b)}}
        \right).
\end{equation}
Set
\[
 A=\frac{1}{s^2\sqrt{ab}},
 \qquad
 B=\frac{D}{s\sqrt{ab(a+b)}}.
\]
Conditional independence across the \(p\) coordinates gives a bound \(C_p(A+B)^p\). Since
\[
        (A+B)^p\le 2^{p-1}(A^p+B^p),
\]
and \(p\) is fixed, the factor \(2^{p-1}\) may be absorbed into \(C_p\). Hence
\begin{equation}\label{eq:candidate-survival}
 \Prob(\text{the pair remains bad})
 \le C_p\left(
        \frac{1}{s^{2p}(ab)^{p/2}}
        +\frac{D^p}{s^p[ab(a+b)]^{p/2}}
 \right).
\end{equation}

\emph{Expected number of bad pairs.}
The two sums
\begin{equation}\label{eq:summable-series}
 \sum_{a,b\ge1}\frac{1}{(ab)^{p/2}}
 \quad\text{and}\quad
 \sum_{a,b\ge1}
 \frac{\gcd(a,b)^p}{[ab(a+b)]^{p/2}}
\end{equation}
are finite for $p\ge3$. For the second, write $a=Du$ and $b=Dv$, where $D=\gcd(a,b)$, and then discard the coprimality condition:
\[
 \sum_{a,b\ge1}
 \frac{\gcd(a,b)^p}{[ab(a+b)]^{p/2}}
 \le
 \sum_{D\ge1}D^{-p/2}
 \sum_{u,v\ge1}[uv(u+v)]^{-p/2}<\infty.
\]
Indeed, $u+v\ge2\sqrt{uv}$, so the inner summand is at most a constant times $(uv)^{-3p/4}$.

Let \(Z\) be the number of bad pairs in \(Y\). Summing \eqref{eq:candidate-survival} over at most \(N\) starting positions for each \(a,b\), and using the convergence of the two displayed series, gives
\begin{equation}\label{eq:expected-bad-pairs}
        \E \left[Z\right]
        \le C_pN(s^{-p}+s^{-2p})
        \le C_p\varepsilon_p(1+s^{-p}).
\end{equation}
Choose \(\varepsilon_p\) so that the last quantity is less than \(1\). Since \(Z\) is a nonnegative integer, some outcome has \(Z=0\), and the corresponding word \(Y\) is 3-free.

\emph{Retaining a frequent letter.}
Since $M\mid N$, write $N=\ell M$. Before splitting, this prefix contains exactly \(\ell s=Ns/M\) occurrences of \(\sigma\). At most two \(\sigma\)-blocks meeting the prefix are incomplete, and every complete block contains at least $\delta_pg$ occurrences of $\sigma_1$. Therefore
\begin{equation}\label{eq:output-density}
        |Y|_{\sigma_1}
        \ge \delta_p(\ell s-2g).
\end{equation}
The floor in \eqref{eq:amplifier-length} is at least \(2\), so \(\ell\ge2\). Since \(K\ge4\) and \(g\le s/K\), this implies \(\ell s\ge4g\). Also \(M\le2n\) and \(s\ge\rho n\), so \eqref{eq:output-density} gives
\[
        |Y|_{\sigma_1}
        \ge c_p\cdot\frac{Ns}{M}
        \ge c_p\rho N.
\]
Thus we may take $\tau=\sigma_1$.

Finally, replacing \(\sigma\) by \(p+1\) letters increases the alphabet size by \(p\), and the separator adds one more letter. Thus the net increase is \(p+1\). The lower bound in \eqref{eq:amplifier-conclusion} follows from \eqref{eq:N-bounds} and \(s\ge\rho n\), after decreasing the fixed \(c_p\) if necessary.
\end{proof}

\section{Iteration}\label{sec:iteration}

The amplification proposition immediately yields a family of double-exponential bounds.

\begin{proposition}\label{prop:general-rate}
For every integer $p\ge3$,
\begin{equation}\label{eq:general-rate}
        \log_2\log_2 L(d)
        \ge \frac{\log_2p}{p+1}\cdot d-O_p(1)
\end{equation}
for all sufficiently large \(d\), where the implicit constant in \(O_p(1)\) may depend on \(p\).
\end{proposition}

\begin{proof}
Fix $p\ge3$ and the constants from \cref{prop:amplification}. By \cref{prop:seed}, choose a marked 3-free word \((X_0,\sigma_0,\rho_0)\), with \(\rho_0=1/2\), whose length is a sufficiently large constant depending only on \(p\). Let
\[
        n_t=|X_t|,
        \qquad
        A_t=\log_2 n_t,
\]
and let \(r_t\) be the number of letters used by \(X_t\).

At each stage, apply \cref{prop:amplification} and set \(\rho_{t+1}=c_p\rho_t\). Thus \(\rho_t=\frac12c_p^t\), and
\begin{equation}\label{eq:iteration-recurrence}
        n_{t+1}\ge c_p\rho_t^pn_t^p,
        \qquad
        r_{t+1}\le r_t+p+1.
\end{equation}
Taking logarithms of the first inequality gives, for a constant \(B_p>0\),
\begin{equation}\label{eq:log-recurrence}
        A_{t+1}\ge pA_t-B_p(t+1).
\end{equation}
Also
\[
        \log_2(C_p\rho_t^{-2})=\alpha_pt+\beta_p
\]
for constants \(\alpha_p,\beta_p\). Choose \(H_p\) so that
\[
        H_p\ge\max\{\alpha_p,\beta_p\}
        \quad\text{and}\quad
        (p-2)H_p\ge B_p,
\]
and choose the initial word so that \(A_0\ge H_p\). If \(A_t\ge H_p(t+1)\), then \eqref{eq:log-recurrence} gives
\[
 A_{t+1}\ge [pH_p-B_p](t+1)\ge H_p(t+2).
\]
Thus \(A_t\ge H_p(t+1)\) for every \(t\), which in turn implies \(n_t\ge C_p\rho_t^{-2}\) at every stage.

Iterating \eqref{eq:log-recurrence} more precisely gives
\begin{align*}
        A_t
        &\ge p^t\left(
        A_0-B_p\sum_{j=0}^{t-1}\frac{j+1}{p^{j+1}}
        \right)\\
        &\ge \gamma_pp^t,
\end{align*}
for some \(\gamma_p>0\), after increasing the fixed initial length once more. Also
\[
        r_t\le r_0+(p+1)t.
\]
Given sufficiently large $d$, take
\[
        t=\floor{\frac{d-r_0}{p+1}}.
\]
The word \(X_t\) uses at most \(d\) letters, so after relabeling its alphabet as a subset of \([d]\) we have \(L(d)\ge n_t\). Hence
\[
        \log_2\log_2L(d)
        \ge \log_2A_t
        \ge t\log_2p-O_p(1),
\]
which is \eqref{eq:general-rate}.
\end{proof}

For completeness, let \(f(x)=\log_2x/(x+1)\). Up to the positive factor \(1/\log 2\), the sign of \(f'(x)\) is the sign of
\[
        1+\frac1x-\log x.
\]
This is negative for \(x\ge4\), while \(f(4)>f(3)\). Hence \(p=4\) is the unique maximizing integer \(p\ge3\).

\begin{proof}[Proof of \cref{thm:main}]
Apply \cref{prop:general-rate} with $p=4$. Since $\log_2 4=2$, this gives
\[
        \log_2\log_2 L(d)\ge \frac{2}{5}d-O(1),
\]
which is \eqref{eq:main-loglog}. 
\end{proof}

\section{Finite Planar Walks with a Small Step Set}\label{sec:planar}

For a finite set \(S\subseteq\Z^2\), an \(S\)-walk is a sequence whose successive differences belong to \(S\). The main theorem gives the following quantitative finite consequence.

\begin{theorem}[Finite planar walks with a small step set]\label{thm:swalk-projection}
There is an absolute constant \(C>0\) such that, for every sufficiently large \(n\), there is a finite set
\begin{equation}\label{eq:step-set-box}
        S\subseteq
        \{1\}\times\{1,\ldots,(n+1)^3\}
        \subseteq\Z^2
\end{equation}
with
\begin{equation}\label{eq:step-set-size}
        |S|\le \frac{5}{2}\log_2\log_2 n+C
\end{equation}
and an \(S\)-walk of \(n\) steps whose vertices are distinct and contain no three collinear points.
\end{theorem}

\begin{proof}
Let
\[
        d=\left\lceil \frac52\bigl(\log_2\log_2 n+C_0\bigr)\right\rceil.
\]
For sufficiently large \(n\), \cref{thm:main} applies and gives \(L(d)\ge n\). Taking a prefix of length \(n\), we obtain a 3-free word \(x_1\cdots x_n\) over \([d]\). Moreover,
\begin{equation}\label{eq:projection-dimension}
        d\le \frac{5}{2}\log_2\log_2 n+C
\end{equation}
for an absolute constant \(C\).

For adjacent nonempty intervals \(I,J\), let \(\mathbf c(I),\mathbf c(J)\in\Z^d\) be their letter-count vectors and define
\begin{equation}\label{eq:projection-obstruction}
        R(I,J)=|J|\mathbf c(I)-|I|\mathbf c(J).
\end{equation}
The word is 3-free, so \(R(I,J)\ne0\). There are exactly \(\binom{n+1}{3}\) adjacent pairs, and this number is smaller than \(K=(n+1)^3\).

Choose \(w_1,\ldots,w_d\) independently and uniformly from \(\{1,\ldots,K\}\). For a fixed nonzero vector \(R\in\Z^d\), choose a coordinate \(k\) with \(R_k\ne0\) and condition on all \(w_i\) with \(i\ne k\). At most one value of \(w_k\) satisfies \(R\cdot w=0\). Hence
\[
        \Prob(R\cdot w=0)\le\frac1K.
\]
A union bound over the \(\binom{n+1}{3}\) vectors in \eqref{eq:projection-obstruction} shows that some \(w\in\{1,\ldots,K\}^d\) satisfies
\begin{equation}\label{eq:generic-projection}
        R(I,J)\cdot w\ne0
\end{equation}
for every adjacent pair \(I,J\).

Assign to letter \(i\) the step
\[
        s_i=(1,w_i)\in\Z^2
\]
and put
\[
        Q_t=\sum_{m=1}^t s_{x_m}.
\]
The first coordinate of \(Q_t\) is \(t\), so the vertices are distinct. The two displacement vectors associated with adjacent intervals \(I,J\) are
\[
        (|I|,\mathbf c(I)\cdot w)
        \quad\text{and}\quad
        (|J|,\mathbf c(J)\cdot w).
\]
Their determinant is
\[
        |I|\mathbf c(J)\cdot w-|J|\mathbf c(I)\cdot w
        =-R(I,J)\cdot w,
\]
which is nonzero by \eqref{eq:generic-projection}. Thus the two displacements are not parallel, so no three vertices are collinear. Taking \(S=\{s_1,\ldots,s_d\}\), for which \(|S|\le d\), proves \eqref{eq:step-set-box} and \eqref{eq:step-set-size}.
\end{proof}

\section*{Acknowledgments}

GPT-5.6 Pro was used to assist with optimizing the parameters in the proof. The central amplification construction and the idea of introducing randomness were developed by the author.

\end{document}